\documentclass{amsart}
\usepackage{mystyle}

\newtheorem{Theorem}{Theorem}[section] 

\newtheorem{Definition}{Definition}[section]
\newtheorem{Lemma}{Lemma}[section]

\newtheorem{Remark}{Remark}[section]

\begin{document}
\title[Stochastic Perron's Method]{Stochastic Perron's method and verification without smoothness using viscosity comparison: obstacle problems and Dynkin games}

\author{ Erhan Bayraktar}
\address{University of Michigan, Department of Mathematics, 530 Church Street, Ann Arbor, MI 48109.} \email{erhan@umich.edu.}
 \thanks{The research of E. Bayraktar was supported in part by the National Science Foundation under grants DMS 0906257, DMS 0955463, and DMS 1118673.}

 \author{Mihai S\^{\i}rbu}
 \address{University of Texas at Austin,
    Department of Mathematics, 1 University Station C1200, Austin, TX,
    78712.}  
    \email{sirbu@math.utexas.edu.} 
    \thanks{The research of
    M. S\^{\i}rbu was supported in part by the National Science
    Foundation under Grant
    DMS 0908441.} \thanks{Any opinions, findings, and conclusions or recommendations expressed in this material are those of the authors and do not necessarily reflect the views of the National Science Foundation.}
    
\date{\today}

\keywords{Perron's method, viscosity solutions, non-smooth verification, comparison principle}
  
\subjclass[2010] {Primary
60G40, 60G46,  60H30;  Secondary  35R35, 35K65, 35K10}

\begin{abstract}
We  adapt  the Stochastic Perron's method  in \cite{bayraktar-sirbu-1} to the  case of double obstacle problems associated to Dynkin games. We construct, symmetrically, a  viscosity sub-solution which dominates the upper value of the game and a viscosity super-solution lying below the lower value of the game.
If the double obstacle problem satisfies the viscosity comparison property, then the game has a value which is equal to the unique and continuous viscosity solution. In addition, the optimal strategies of the two players are equal to the first hitting times of the two stopping regions, as expected.  The (single) obstacle problem associated to optimal stopping can be viewed as a very particular case. This is the first instance  of a  non-linear problem where the Stochastic Perron's method can be applied successfully. 
 \end{abstract}

\maketitle 
%
%

 \section{Introduction} 

 In \cite{bayraktar-sirbu-1}, the authors introduce  a stochastic version of  Perron's method to construct viscosity (semi)-solutions for \emph{linear}  parabolic (or elliptic) equations, and use viscosity comparison as a substitute for verification (It\^o's lemma).  
 
 The present note  extends the Stochastic Perron's method to the case of (double) obstacle problems. The  conceptual contribution of the present note lies in \emph{the proper identification of stochastic sub- and super- solutions for the obstacle problem} (see Definitions \ref{def:supersolutions} and \ref{def:subsolutions}). The technical contribution consists in proving that, having identified such a definition of stochastic solutions, the Perron's method actually does produce viscosity super- and sub-solutions. 
Technically, the proofs turn out to be very  different from \cite{bayraktar-sirbu-1}. The short paragraph    before Lemma \ref{lattice} singles out the  difficulty in trying to follow the results in \cite{bayraktar-sirbu-1}.  As the reader can see below, Definitions  \ref{def:supersolutions} and \ref{def:subsolutions} are tailor-made to fit the proofs, thus avoiding considerations related to Markov processes.
 
 The purpose of the present note is to explain how to adapt the Stochastic Perron's method to this very first \emph{non-linear} case.  In order to keep  the presentation short and simple, we therefore use a very similar framework  (and notation) as in \cite{bayraktar-sirbu-1} namely that of:  the state space is the whole $\mathbb{R}^d$,  there is a finite time horizon $T$, there is  no running-cost and no discounting. The obstacle and the terminal pay-off are bounded, in order to avoid any issue related to integrability. However, our method works for more general obstacle problems, including all features assumed away as described above. In particular, the method can be applied to elliptic  obstacle problems rather than parabolic.

Optimal stopping  and the more general problem of optimal stopping games (so called ``Dynkin games'',  introduced in  \cite{dynkin69}), are fundamental problems in stochastic optimization. If the optimal stopping is associated to Markov diffusions, there are  two  classic approaches to solve the problem:

1. The analytic approach consists in writing the Hamilton-Jacobi-Bellman  equation (which takes the form of an obstacle problem), finding a smooth solution and then go over  \emph{verification arguments}. The method works only if the solution to the HJB is smooth enough to apply It\^o's formula along the diffusion. This is particularly delicate if  the diffusion degenerates.

2.  The probabilistic approach consists in the analysis of the value function(s), using heavily the Markov property and conditioning, to show a similar conclusion to the analytic approach: it is optimal to stop as soon as the player(s) reach(es) the contact region between the value function and the obstacle.

We provide here a new look to the the problem of optimal stopping and Dynkin games. Our approach is \emph{a  probabilistic version of the analytic approach}. The probabilistic arguments rely \emph{only} on It\^o's formula (which is applied to smooth test functions), without using the Markov property of solutions of the SDE. The Markov  property is not even assumed.  The fine interplay between how much smoothness is needed for a solution of a PDE in order to apply It\^o's formula along the SDE is hidden behind the definition of \emph{stochastic} super- and sub-solutions, which traces back to the seminal work of Stroock and Varadhan  \cite{Stroock-Varadhan}. We could summarize the message of Theorem \ref{theorem} as: if a viscosity comparison result for the HJB holds, then there is no need to either find a smooth solution of the HJB, or to analyze the value function(s) to solve the optimization problem. Formally, all classic results hold as expected, i.e.,  \emph{the unique continuous (but possibly non-smooth) viscosity solution is equal to the value of the game and it is optimal for the player(s) to stop as soon as they reach their corresponding contact/stopping regions.}

   \section{The set-up and main results}
 
 Fix a time interval $T>0$ and for each $0\leq s<T$ and $x\in \mathbb{R}^d$ consider the stochastic differential equation
 \begin{equation}\label{sde}\left \{
 \begin{array}{l}
 dX_t=b(t,X_t)dt+\sigma (t, X_t)dW_t, \ \ s\leq t\leq T,\\
 X_s=x.
 \end{array}\right.
 \end{equation}
 We assume that  the coefficients $b :[0,T]\times \mathbb{R}^d\rightarrow \mathbb{R}^d$ and $\sigma :[0,T]\times \mathbb{R}^d\rightarrow \mathbb{M}_{d,d'}(\mathbb{R})$ are continuous.  We also assume that, for each $(s,x)$ equation \eqref{sde} has at least a weak  non-exploding solution 
 $$\Big((X^{s,x}_t)_{s\leq t\leq T}, (W^{s,x}_t)_{s\leq t\leq T}, \Omega ^{s,x}, \mathcal{F}^{s,x}, \mathbb{P}^{s,x},( \mathcal{F}^{s,x}_t) _{s\leq t\leq T}\Big),$$
  where the  $W^{s,x}$ is  a $d'$-dimensional Brownian motion on the stochastic basis $$(\Omega ^{s,x}, \mathcal{F}^{s,x}, \mathbb{P}^{s,x},( \mathcal{F}^{s,x}_t) _{s\leq t\leq T})$$ and  the filtration $( \mathcal{F}^{s,x}_t) _{s\leq t\leq T}$  satisfies the  usual conditions.
  We denote by $\mathcal{X}^{s,x}$ the non-empty set of such weak solutions.
 It is well known, for example from \cite{MR2190038}, that  a sufficient condition for the existence of non-exploding solutions, in addition to continuity of the coefficients, is the condition of linear growth:
 $$|b(t,x)|+|\sigma (t,x)|\leq C(1+|x|), \ \ (t,x)\in [0,T]\times \mathbb{R}^d.$$ 
 We emphasize that we do \emph{not} assume uniqueness in law of the weak solution.
 In order to insure that $\mathcal{X}^{s,x}$ is a set in the sense of axiomatic set theory, we restrict to weak solutions where the probability space $\Omega$ is an element of a fixed universal set $\mathcal{S}$ of possible probability spaces.
 For each $(s,x)\in [0,T]\times \mathbb{R}^d$ we \emph{choose} a fixed solution $X^{s,x}$ as above, using the axiom of choice. We do not assume that the selection is Markov.

Let $g:\mathbb{R}^d\rightarrow \mathbb{R}$ be a \emph{bounded} and measurable function  (terminal pay-off). Let also $l, u:[0,T]\times \mathbb{R}^d\rightarrow \mathbb{R}$ be two \emph{bounded} and measurable functions satisfying
$l \leq u.$
The functions $l,u$ are the  lower and the upper obstacles. 
Assume, in addition, that
$$l(T,\cdot)\leq g\leq u(T,\cdot).$$
For each weak solution $X^{s,x}$ we denote by $\mathcal{T}^{s,x}$ the set of stopping times $\tau$ (with respect to the filtration $( \mathcal{F}^{s,x}_t) _{s\leq t\leq T}$) which satisfy $s\leq \tau \leq T$. 
The first  player chooses a stopping time $\rho \in \mathcal{T}^{s,x}$ and the second player chooses a stopping time $\tau \in \mathcal{T}^{s,x}$,  so that the first player \emph{pays} to the second player the amount
$$J(s,x,\tau, \rho):=
\mathbb{E}^{s,x}\left [ \mathbb{I}_{\{ \tau <\rho \} }l(\tau , X^{s,x}_{\tau})+
 \mathbb{I}_{\{ \rho \leq \tau , \rho <T\} })u(\rho, X^{s,x}_{ \rho})+ \mathbb{I}_{\{ \tau =\rho=T\} } g(X^{s,x}_T)\right ].$$
We are now  ready to introduce the \emph{lower value of the Dynkin game}
 $$v_*(s,x):= 
 \sup _{\tau \in \mathcal{T}^{s,x}} \inf _{\rho \in \mathcal{T}^{s,x}} J(s,x,\tau, \rho)$$ 
 and the 
 \emph{upper value of the game}
  $$v^*(s,x):=
 \inf _{\rho \in \mathcal{T}^{s,x}}  \sup _{\tau \in \mathcal{T}^{s,x}} J(s,x,\tau, \rho).$$ 
 The lower and the upper values satisfy
 $$v_*\leq v^*$$ and, if the two functions coincide
 we say that the game has a value.
  \begin{Remark}\label{measurability}{\rm  At this stage, we cannot even conclude that  $v_*$ and $v^*$ are  measurable.}
 \end{Remark}
\emph{If} the selection $X^{s,x}$ is actually Markov, we usually associate to the game of optimal stopping the \emph{nonlinear} PDE (double obstacle problem)

\begin{equation}\label{pde}
\left \{
\begin{array}{l}
F(t,x, v, v_t, v_x, v_{xx})=0, \; \textrm{~on~} [0,T)\times \mathbb{R}^d,\\
u(T,\cdot )=g,
\end{array}\right.\end{equation}
where 
\begin{equation}\label{isaacs}
\begin{split}
F(t,x, v, v_t, v_x, v_{xx}):&= \\
\max \{v-u, \min \{-v_t-L_tv, v-l\}\} 
&=\min  \{v-l,\max\{-v_t-L_tv, v-u\}\},
\end{split}
\end{equation}
and the time dependent operator $L_t$ is defined by
$$(L_t u)(x):=\langle b(t,x),\nabla u(t,x)\rangle +\frac 12 Tr (\sigma (t,x)\sigma ^T (t,x) \nabla _x^2 u(t,x)),\ \ 0\leq t<T,\  x\in \mathbb{R}^d.$$ 
The second equality in \eqref{isaacs} relies on the assumption that the obstacle are ordered, $l\leq u$.
 \subsection{Stochastic Perron's method} The main \emph{conceptual} contribution of the present note is contained below, in  the proper definitions of stochastic super and sub-solutions of the parabolic PDE  \eqref{pde} in the spirit of \cite{Stroock-Varadhan} and following our previous work in \cite{bayraktar-sirbu-1}.
In order to have a workable definition of stochastic semi-solution for the (double) obstacle problem, two very important things have to be taken into account:
\begin{enumerate}
\item  (conceptual) one has to account for the stochastic sub-, super-solutions along the diffusion to be  sub-, super- martingales inside the continuation region(s). However, the diffusion may re-enter the continuation region \emph{after} hitting the stopping region. In other words, the martingale property until the \emph{first hitting time} of the stopping regions \emph{after the initial time} may not be enough.  The definition should start at \emph{stopping times} (possibly subsequent to $s$)  rather than at $s$, where $s$ is the starting time.
\item (technical) semi-solutions are evaluated  at \emph{stopping times}, so  the Optional Sampling Theorem has to be built in the definition of sub- and super- martingales, in case semi-solutions are \emph{less than continuous}.  This  idea can  eliminate an important part of the technicalities in the previous work on the linear case \cite{bayraktar-sirbu-1}, if used directly. However, we choose below a different (and arguably better) route, where semi-solutions of  \eqref{pde} are   \emph{continuous}. This can be achieved due to the technical Lemma \ref{Dini} resembling Dini's Theorem. Continuous semi-solutions together with  Dini's Theorem can also eliminate the technical difficulties in \cite{bayraktar-sirbu-1}. 
\end{enumerate}
 \begin{Definition}\label{def:supersolutions}
The set of stochastic super-solutions for the parabolic PDE \eqref{pde}, denoted by $\mathcal{V}^+$, is  the set of functions $v:[0,T]\times \mathbb{R}^d\rightarrow \mathbb{R}$ which have the following properties
\begin{enumerate}
\item are continuous (C) and  bounded on $[0,T]\times \mathbb{R}^d$. In addition,  they satisfy 
$v\geq l$ and
the terminal condition   $v(T,\cdot )\geq g$.

\item 
 for each $(s,x)\in[0,T]\times \mathbb{R}^d$, and any stopping time $\tau _1\in \mathcal{T}^{s,x}$, the function $v$ along the solution of the SDE is a super-martingale in between $\tau _1$ and the first  (after $\tau _1$) hitting time of the upper-continuation region $\mathcal{S}^+ (v):=\{v\geq u\}$. 
More precisely, for any  
 $\tau _1\leq \tau_2\in \mathcal{T}^{s,x},$ we have  
$$v(\tau _1, X^{s,x}_{\tau _1})\geq 
\mathbb{E}^{s,x}\left [ v(\tau _2\wedge \rho ^+, X^{s,x}_{\tau _2\wedge \rho ^+})|\mathcal{F}^{s,x}_{\tau _1 }
\right]\ - \mathbb{P}^{s,x}\ a.s.$$
where the stopping time $\rho ^ +$ is defined as
\[
\begin{split}
\rho ^ +(v, s,x, \tau_1):&=\inf\{t \in [\tau_1, T]:v(t,X^{s,x})\geq u(t,X^{s,x})\}
\\&=\inf \{t \in [\tau_1,T]:X^{s,x}_t \in \mathcal{S}^+(v)\}.
\end{split}
\]
\end{enumerate}
\end{Definition}
\begin{Remark}\label{rem:supersolution} {\rm The super-solution property means that, starting at \emph{any} stopping time $\tau _1\geq s$, the function along the diffusion is a super-martingale 
\emph{before} hitting the upper stopping region
$\mathcal{S}^+ (v):=\{v\geq u\}.$
This takes into account the fact that the stopping time $\tau _1$ may be subsequent to the first time the diffusion enters the stopping region $\mathcal{S}^+(v)$ after $s$. In other words, it accounts for the possibility to re-enter the continuation region. Building in the Optional Sampling Theorem (i.e. considering the later time $\tau _2$ stochastic) does not make a difference, since $v$ is continuous. However, a definition  involving stopping times had to be considered anyway, since the starting time $\tau _1$ is stochastic.}
\end{Remark}
In order to simplify the presentation, we would like to explain the notation:
\begin{enumerate}
\item usually, stopping times relevant to the first (minimizer) player will be denoted by $\rho$ (with different indexes or superscripts) and  stopping times relevant to the second (maximizer) player will be denoted by $\tau$ (with different indexes or superscripts)
\item the superscripts $+$ and $-$ refer to hitting the upper and the lower obstacles, respectively.
\end{enumerate}
\begin{Definition}\label{def:subsolutions}
The set of stochastic sub-solutions for the parabolic PDE \eqref{pde}, denoted by $\mathcal{V}^-$, is  the set of functions $v:[0,T]\times \mathbb{R}^d\rightarrow \mathbb{R}$ which have the following properties
\begin{enumerate}
\item are continuous (C) and  bounded on $[0,T]\times \mathbb{R}^d$. In addition,  they satisfy 
$v\leq u $ and
the terminal condition   $v(T,\cdot )\leq g$.

\item 
  for each $(s,x)\in[0,T]\times \mathbb{R}^d$, and any stopping time
  $\rho _1\in \mathcal{T}^{s,x}$, the function $v$ along the solution of the SDE is a sub-martingale in between $\rho_1$ and the first (after $\rho _1$)
 hitting time of the lower continuation region $\mathcal{S}^- (v):=\{v\leq l\}.$
 More precisely, for each  
  $\rho _1\leq \rho_2\in \mathcal{T}^{s,x},$ we have 
$$v(\rho _1, X^{s,x}_{\rho _1
})\leq 
\mathbb{E}^{s,x}\left [ v(\rho _2\wedge \tau  ^-, X^{s,x}_{\rho _2\wedge \tau  ^-})|\mathcal{F}^{s,x}_{\rho ^1}
\right]\ -\mathbb{P}^{s,x}\ a.s.$$
where the stopping time $\tau ^- $ is defined as
\[
\begin{split}
\tau ^-(v, s,x, \rho _1):&=\inf \{t \in [\rho _1, T]: v(t,X^{s,x})\leq l(t,X^{s,x})\}
\\&=\inf \{t \in [\rho _1, T]:X^{s,x} \in \mathcal{S}^-(v)\}.
\end{split}
\]
\end{enumerate}
\end{Definition}
An identical  comment to Remark \ref{rem:supersolution} applies to sub-solutions. The next lemma is actually obvious:

\begin{Lemma}\label{lem:nonempty}
Assume $g$, $l$ and $u$ are bounded. Then 
  $\mathcal{V}^+$ and $\mathcal{V}^-$ are nonempty.
  
\begin{Remark}\label{unbounded}{\rm 
We  decided to work in the framework of bounded $l$, $u$ and $g$ to minimize technicalities related to integrability. However, our method works in more general situations. If $l$ and $u$ are assumed unbounded, then we need the additional technical assumptions
\begin{enumerate}
\item $\mathcal{V}^+$ and $\mathcal{V}^-$ are nonempty. This is always the case if $l$ and $u$ are actually continuous.
\item  for each $(s,x)$ we need that 
$$ \mathbb{E}^{s,x}\left [ \sup _{0\leq  t \leq T} \Big (l(t, X^{s,x}_t)|+|u(t, X^{s,x}_t)| \Big )+|g(X^{s,x}_t)|\right]<\infty.$$
This is the usual assumption made in optimal stopping or Dynkin games.
\end{enumerate}}
\end{Remark}

\end{Lemma}
The next result
is quite clear and represents a counterpart to something similar in \cite{bayraktar-sirbu-1}. However, it  needs assumptions on the semi-continuity of the obstacles.

\begin{Lemma}\label{lem:verification} 
\begin{enumerate}
\item  If $u$ is lower semi-continuous (LSC) then for each $v\in \mathcal{V}^+$ we have $v\geq v^*$.
\item  If $l$ is upper semi-continuous (USC) then for each $v\in \mathcal{V}^-$ we have $v\leq v_*$.
\end{enumerate}
\end{Lemma}
\begin{proof} We only prove part (ii) since the other part is symmetric. 
Using the LSC of $v$  (since $v$ is actually continuous) and the USC of $l$, we conclude that the set
$\mathcal{S}^- (v)=\{v\leq l\}$ is closed. This means that,  denoting by 
$\tau ^ - := \tau ^- (v;s,x,s)$, we have  
$$\mathbb{I}_{\{\tau ^- <\rho \}}v(\tau ^-, X^{s,x}_{\tau ^-})
\leq  \mathbb{I}_{\{\tau ^- <\rho \}}l(\tau ^-, X^{s,x}_{\tau ^-})$$
 for \emph{each} $\rho \in \mathcal{T}^{s,x}$. Applying the definition of Stochastic sub-solutions between the times
$$\rho _1:=s\leq \rho _2:=\rho,$$
together with the definition of the cost $J$ and the fact that $v\leq u$ we obtain that
$$v(s,x)\leq \mathbb{E}^{s,x}[v(\tau ^-\wedge \rho, X^{s,x}_{\tau ^-\wedge \rho})]\leq J(s,x, \tau ^-,\rho).$$
Taking the inf over $\rho$, we conclude that
$$v(s,x)\leq \inf _{\rho \in \mathcal{T}^{s,x}}J(s,x, \tau ^-,\rho)\leq v_*(s,x).$$
\end{proof}
We assume now that $l$ is USC and $u$ is LSC.  Following \cite{bayraktar-sirbu-1} and using Lemmas
\ref{lem:nonempty} and    \ref{lem:verification},  we  define
$$ v^-:= \sup _{v\in \mathcal{V}^-}v\leq v_*\leq v^*\leq v^+:=\inf _{w\in \mathcal{V}^+}w.$$
The next result is actually the main result of the paper.

\begin{Theorem}{\rm (Stochastic Perron's Method)}\label{Perron}
\begin{enumerate}

\item Assume  that $g$, $l$ are USC and $u$ is LSC.
Assume, in addition that
\begin{equation}\label{u}
\textrm{there~exists~}v\in \mathcal{V}^+\textrm{~such ~that~} v\leq u.
\end{equation}
Then,   $v^+$ is a bounded and  USC  viscosity sub-solution of 
\begin{equation}\label{pde1}
\left \{
\begin{array}{l}
F(t,x, v, v_t, v_x, v_{xx}) \leq 0 \textrm{~on~} [0,T)\times \mathbb{R}^d,\\
v(T,\cdot)\leq g.
\end{array}\right.\end{equation}

\item 
Assume  $g$,  $u$ are  LSC and $l$ is USC. 
Assume, in addition that
\begin{equation}\label{l}
\textrm{there~exists~}v\in \mathcal{V}^-\textrm{~such ~that~} v\geq l.
\end{equation}
Then, 
$v^-$ is a bounded and LSC viscosity  super-solution of 
\begin{equation}\label{pde2}
\left \{
\begin{array}{l}
F(t,x, v, v_t, v_x, v_{xx})\geq 0 \textrm{~on~} [0,T)\times \mathbb{R}^d,\\
v(T,\cdot)\geq g.
\end{array}\right.\end{equation}
\end{enumerate} \end{Theorem}
\begin{Remark}\label{rem:nonempty}{\rm 
Assumption \eqref{u} is satisfied if the upper obstacle $u$ is continuous, and Assumption \eqref{l} is satisfied if the lower obstacle $l$ is continuous.}
\end{Remark}
We have $v^+(T,\cdot )\geq g$ 
and  $v^-(T,\cdot)\leq g$
by construction.  Therefore, the terminal conditions in \eqref{pde1} and \eqref{pde2} can be replaced by equalities.
 The proof of the Theorem  \ref{Perron} is  technically different from the proof in \cite{bayraktar-sirbu-1}. The main reason is that, we cannot prove directly that
 $$ v^+\in \mathcal{V}^+,\ \ \ v^-\in \mathcal{V}^-,$$
 even if we weaken the continuity of super-solutions in $\mathcal{V}^+$ to USC and the continuity of sub-solutions in $\mathcal{V}^-$ to LSC.
This technical hurdle  is circumvented by the weaker lemma below, together with an approximation argument in Lemma~\ref{Dini}.

\begin{Lemma}\label{lattice}
\begin{enumerate}
\item Assume $u$ is LSC. If $v^1, v^2\in \mathcal{V}^+$, then $v^1\wedge v^2 \in \mathcal{V}^+.$
\item Assume $l$ is USC. If $v^1, v^2\in \mathcal{V}^-$, then $v^1\vee v^2 \in \mathcal{V}^-.$
\end{enumerate}
\end{Lemma}
\begin{proof} We prove item (ii). It is easy to see that  $ v:=v^1\vee v^2$ is continuous as well as
$v\leq u$ and $v(T,\cdot)\leq g$. The only remaining thing to prove is
 for each $(s,x)\in[0,T]\times \mathbb{R}^d$, and any stopping times
$$\rho _1\leq \rho_2\in \mathcal{T}^{s,x},$$ 
\begin{equation}\label{v12}v(\rho _1, X^{s,x}_{\rho _1
})\leq 
\mathbb{E}^{s,x}\left [ v(\rho _2\wedge \tau  ^-, X^{s,x}_{\rho _2\wedge \tau  ^-})|\mathcal{F}^{s,x}_{\rho _1}
\right]\ -\mathbb{P}^{s,x}\ a.s.
\end{equation}
where the stopping time $\tau ^-$ is defined by $\tau ^-=\tau ^- (v; s,x, \rho ^1)$. In other words, we want to prove the sub-martingale property of $v$ along the process $X^{s,x}$ in between the stopping times
$\rho_1\leq \rho _2\wedge \tau ^-$.
The idea of the proof relies on a sequence of stopping times  $(\gamma _n)_{n\geq 0}$ defined recursively as follows: set  $\gamma _0=\rho _1$
and then, for each $n=0,1,2,\dots$
\begin{enumerate}
\item if  
$v(\gamma _n, X_{\gamma _n})\leq l(\gamma _n, X_{\gamma _n}),$
then $\gamma _{n+1}:=\gamma _n$ 
\item if  
$v(\gamma _n, X_{\gamma _n})=v^1(\gamma _n, X_{\gamma _n})>
 l(\gamma _n, X_{\gamma _n}),$
then $$\gamma _{n+1}:=\inf\{t \in [\gamma _n,T]:\{v^1 (t, X_t)\leq l(t, X_t)\}.$$ 
In this case, we note that 
$v^1(\gamma _{n+1}, X_{\gamma _{n+1}})\leq v(\gamma _{n+1}, X_{\gamma _{n+1}})$ and  \\
$v^1(\gamma _{n+1}, X_{\gamma _{n+1}})\leq l(\gamma _{n+1}, X_{\gamma _{n+1}})$.
\item if  
$v(\gamma _n, X_{\gamma _n})=v^2(\gamma _n, X_{\gamma _n})>
 l(\gamma _n, X_{\gamma _n}),$ then 
 $$\gamma _{n+1}:=\inf\{t \in [\gamma _n, T]: v^2 (t, X_t)\leq l(t, X_t)\}.$$  
In this case, we note that 
$v^2(\gamma _{n+1}, X_{\gamma _{n+1}})\leq v(\gamma _{n+1}, X_{\gamma _{n+1}})$
and \\
$v^2(\gamma _{n+1}, X_{\gamma _{n+1}})\leq l(\gamma _{n+1}, X_{\gamma _{n+1}})$.
\end{enumerate}
One can use  the  definition of stochastic sub-solution for $v^1$ in between the times $\gamma _n \leq \gamma _{n+1}$, together with the observation following item (ii),  or   the definition of stochastic sub-solution for $v^2$ in between $\gamma _n\leq \gamma _{n+1}$  and the observation following item (iii), to conclude that,
 for any $n\geq 0$, $v(t,X_t)$ satisfies the sub-martingale property 
  in between
 $\gamma _n\leq \gamma _{n+1}.$ Concatenating, we conclude that, for each $n$, we have
 \begin{equation}\label{submartingale-n}
 v(\rho _1, X^{s,x}_{\rho _1
})\leq 
\mathbb{E}\left [ v(\rho _2\wedge \gamma  _n, X^{s,x}_{\rho _2\wedge \gamma  _n})|\mathcal{F}^{s,x}_{\rho _1}
\right].
\end{equation}
Now, care must be taken in order to pass to the limit as $n\rightarrow \infty$. By construction, it is clear that $\gamma _n\leq \tau ^-$. 
On the event 
$$A:=\{(\exists)\  n_0,\ \ \gamma _n=\tau  ^-, \ n\geq n_0\},$$
it is very easy to pass to the limit, since the sequence is constant eventually. 
However, on the complementary event 
$$B:=\{\gamma _n<\tau  ^-,\ \ (\forall)\ n\},$$
we have to be more careful. 
Depending on parity, for each $\omega \in B$, there exist
$n_0=n_0(\omega)$ such that
$$v^1(\gamma _{n_0+2k}, X^{s,x}_{\gamma _{n_0+2k}})\leq 
l(\gamma _{n_0+2k}, X^{s,x}_{\gamma _{n_0+2k}}) \textrm{~for~} k\geq 0$$
and 
$$v^2(\gamma _{n_0+2k+1}, X^{s,x}_{\gamma _{n_0+2k+1}})\leq 
l(\gamma _{n_0+2k+1}, X^{s,x}_{\gamma _{n_0+2k+1}}) \textrm{~for~} k\geq 0.$$
This comes from the very definition of the sequence $\gamma _n$.
Since both $v^1$ and $v^2$ are LSC (actually continuous) and $l$ is USC, we can pass to the limit in both inequalities above to concluded that 
for 
$\gamma _{\infty}:=\lim _n \gamma _n$ we have
$$v(\gamma _{\infty}, X^{s,x}_{\gamma _{\infty}})\leq l(\gamma _{\infty}, X^{s,x}_{\gamma _{\infty}}) \textrm{~on~}B.$$
To begin with, this shows that $\gamma _{\infty}\geq \tau  ^-$ so $\gamma _{\infty}=\tau  ^-$ on $B$. However, $\gamma _{\infty}=\tau  ^-$  as well  on $A$.
Now, we  let $n\rightarrow \infty$ in \eqref{submartingale-n} using the Bounded Convergence Theorem, to finish the proof.

This proof uses essentially the continuity of $v^1$ and $v^2$.  However,  using more technical arguments, one could prove something similar assuming only the LSC property for stochastic sub-solutions.

\end{proof}

\noindent \emph{Proof of Theorem \ref{Perron}}. 

We will only prove that $v^+$ is a sub-solution of \eqref{pde1}: the other part is symmetric. 

\noindent {\bf Step 1}. \emph{The interior sub-solution property.} Note that we already know that $v^+$ is bounded and upper semi-continuous (USC). Let 
$$\varphi:[0,T]\times \mathbb{R}^d\rightarrow \mathbb{R}$$ be a $C^{1,2}$-test function function and assume that 
$v^+-\varphi$ attains a strict local maximum (an assumption which is not restrictive) equal to zero at some interior point $(t_0, x_0)\in (0,T)\times \mathbb{R}^d$. Assume that $v^+$ does not satisfy the viscosity sub-solution property, i.e.
$$\min \{\varphi (t_0, x_0)-l(t_0, x_0),\max\{-\varphi (t_0, x_0)_t-L_t\varphi (t_0, x_0), \varphi (t_0, x_0)-u(t_0,x_0)\}\}>0.$$
 According to Assumption \eqref{u} (and this is actually the only place where the assumption is used), we have $v^+\leq u$. This means  
$$v^+(t_0,x_0)>l(t_0,x_0),$$ and
$$-\varphi _t(t_0,x_0)-L_t \varphi (t_0, x_0)>0.$$
Since the coefficients of the SDE  are continuous, we conclude that there exists a small enough ball
$B(t_0, x_0, \varepsilon)$ such that
$$-\varphi _t-L_t \varphi >0\textrm{~on~} \overline{ B(t_0, x_0, \varepsilon)},$$
and 
$$\varphi > v^+\textrm{~on~} \overline{ B(t_0, x_0, \varepsilon)}-(t_0,x_0).$$
In addition, since $\varphi(t_0,x_0)=v^+(t_0,x_0)>l(t_0,x_0)$ and $\varphi $ is continuous and $l$ is USC, we conclude that, if $\varepsilon$ is small enough, then
$$\varphi -\varepsilon\geq l \textrm{~on~}  \overline{ B(t_0, x_0, \varepsilon)}.$$
Since $v^+-\varphi $ is upper semi-continuous and $\overline{B(t_0, x_0, \varepsilon)}-B(t_0, x_0, \varepsilon /2)$ is  compact, this means that there exist a $\delta >0$ such that 
$$\varphi -\delta \geq  v^+\textrm{~on~} \overline{B(t_0, x_0, \varepsilon)}-B(t_0, x_0, \varepsilon /2).$$

The next Lemma is the fundamental step in the proof by approximation (and represents a major technical difference compared to the previous paper \cite{bayraktar-sirbu-1}). This is the result that actually allows us to work with stochastic semi-solutions which are \emph{continuous}.
\begin{Lemma}\label{Dini} Let $0<\delta '<\delta.$ Then there exists a stochastic super-solution $v\in \mathcal{V}^+$ such that 
$$\varphi -\delta ' \geq  v \textrm{~on~} \overline{B(t_0, x_0, \varepsilon)}-B(t_0, x_0, \varepsilon /2).$$\end{Lemma}
\noindent \emph{Proof of Lemma ~\ref{Dini}}. 
Using Lemma \ref{lattice} together with the result in the Appendix of \cite{bayraktar-sirbu-1}, 
we can choose a \emph{decreasing} sequence $(v_n)_{n\geq 0}\subset \mathcal{V}^+$ of stochastic super-solutions such that
$$v_n\searrow v^+.$$
Now the proof follows the idea of Dini's Theorem. More precisely, we denote by 
$$A_n:=\{v_n\geq \varphi -\delta '\}\cap \Big (\overline{B(t_0, x_0, \varepsilon)}-B(t_0, x_0, \varepsilon /2)\Big).$$
We have that $A_{n+1}\subset A_n$ and 
$\cap _{n=0}^{\infty}A_n=\emptyset.$
In addition, since $v_n$ is USC (being continuous) and $\varphi$ is continuous as well, each $A_n$ is closed. 
By compactness, we get that there exits an $n_0$ such that 
$A_{n_0}=\emptyset ,$
which means that
$$\varphi -\delta '>v_{n_0}, \textrm{~on~} \overline{ B(t_0, x_0, \varepsilon)}-(t_0,x_0).$$
We now choose
$v:=v_{n_0}.$
\hfill $\square$ 

We finish the proof of the main theorem as follows.
Let $v\in \mathcal{V}^+$ be given by Lemma \ref{Dini}.
 Choosing $0<\eta < \delta' \wedge \varepsilon$ small enough we have that the function 
$$\varphi _{\eta}:=\varphi-\eta$$ satisfies the properties
$$-\varphi ^{\eta}_t-L_t \varphi  ^{\eta} >0\textrm{~on~} \overline{ B(t_0, x_0, \varepsilon)},$$
$$\varphi ^{\eta} >v \textrm{~on~} \overline{B(t_0, x_0, \varepsilon)}-B(t_0, x_0, \varepsilon /2),$$
$$\varphi ^{\eta} \geq l \textrm{~on~}  \overline{ B(t_0, x_0, \varepsilon)},$$
and 
$$\varphi ^{\eta}(t_0,x_0)=v^+(t_0,x_0)-\eta.$$

Now, we define the new function 
$$v^{\eta}=
\left \{
\begin{array}{l}
 v\wedge  \varphi ^{\eta} \textrm{~on~} \overline{ B(t_0, x_0, \varepsilon)},\\
v \textrm{~outside~}\overline{ B(t_0, x_0, \varepsilon)}.
\end{array}
\right.
$$
We clearly have $v^{\eta}$ is continuous and $v^{\eta}(t_0,x_0)=\varphi ^{\eta}(t_0,x_0)<v^+(t_0, x_0).$ Also, $v^{\eta}$ satisfies the terminal condition (since $\varepsilon$ can be chosen so that $T>t_0+\varepsilon$ and $v$ satisfies the terminal condition).  It only remains to show that 
$v^{\eta}\in \mathcal{V}^+$ to obtain a contradiction.

 We need to show that the process $(v^{\eta} (t,X^{s,x}_t))_{s\leq t\leq T}$ is a super-martingale on $(\Omega ^{s,x}, \mathbb{P}^{s,x})$ with respect to the filtration 
$( \mathcal{F}^{s,x}_t) _{s\leq t\leq T}$ in the \emph{upper continuation region} $C^+:=\{v^{\eta}<u\}$, i.e. satisfies item (ii) in Definition \ref{def:supersolutions}.

The sketch of how to prove this can be given in three steps:

1. the process $(v^{\eta} (t,X^{s,x}_t))_{s\leq t\leq T}$  is a super-martingale locally in the region \Big ($[s,T]\times \mathbb{R}^d-B(t_0, x_0, \varepsilon/2)\Big) \cap C^+$ because it coincides there with the process  $(v (t,X^{s,x}_t))_{s\leq t\leq T}$ which is a super-martingale in $C^+$ (in the sense of the Definition \ref{def:supersolutions}).

2.  In addition, in the region $\overline{B(t_0, x_0, \varepsilon)}$ the process $(v^{\eta} (t,X^{s,x}_t))_{s\leq t\leq T}$ is the minimum between a local super-martingale ($\varphi ^{\eta}$) and a local super-martingale as long as $v<u$.
One needs to perform an \emph{identical argument} to the sequence of stopping times in the proof of Lemma \ref{lattice} to get that we have a  super-martingale in  $\overline{B(t_0, x_0, \varepsilon)}$.

3. The two items above can be easily concatenated as in the proof of Theorem 2.1 in \cite{bayraktar-sirbu-1}. More precisely, we exploit the fact that the two regions in items 1 and 2 overlap over  the ``strip''
$\overline{B(t_0, x_0, \varepsilon)}-B(t_0, x_0, \varepsilon /2)$, so the concatenating sequence is the sequence of ``up and down-crossings'' of this ``strip''.
This concludes the proof of the interior sub-solution property.

{\bf Step 2.} \emph{The terminal condition}. Assume that, for some $x_0\in \mathbb{R}^d$ we have 
$v^+(T,x_0)>g(x_0)\geq l(T,x_0).$
We want to use this information in a similar way to Step 1 to construct a contradiction. Since $g$ and $l$ are USC , there exists an $\varepsilon >0$ such that
$$\max\{l(t,x), g(x)\}\leq v^+(T,x_0)-\varepsilon \ \textrm{~if~} \ \max\{|x-x_0|,\  T-t \}\leq \varepsilon.$$
We now use the fact that $v^+$ is USC to conclude it is bounded above on the compact set
$$\overline{B(T,x_0,\varepsilon)}-B(T,x_0,\varepsilon /2))\cap( [0,T]\times \mathbb{R}^d).$$
Choose $\eta>0 $ small enough so that
\begin{equation}
\label{bound1}
v^+(T,x_0)+\frac{\varepsilon^2}{4\eta} > \varepsilon +\sup_{(t,x)\in (\overline{B(T,x_0,\varepsilon)}-B(T,x_0,\varepsilon /2))\cap( [0,T]\times \mathbb{R}^d)} v^+(t,x).
\end{equation}
Using a (variant of) Lemma \ref{Dini} we can find a $v\in \mathcal{V}^+$ such that

\begin{equation}
\label{bound}
v^+(T,x_0)+\frac{\varepsilon^2}{4\eta} > \varepsilon +\sup_{(t,x)\in (\overline{B(T,x_0,\varepsilon)}-B(T,x_0,\varepsilon /2))\cap( [0,T]\times \mathbb{R}^d)} v(t,x).
\end{equation}

We now define, for $k>0$ the following function
$$\varphi^{\eta,\varepsilon, k}(t,x)=v^+(T,x_0)+\frac{|x-x_0|^2}{\eta}+ k(T-t). $$
For $k$ large enough (but no smaller than 
 $\varepsilon/2\eta$),  we have that 
$$-\varphi^{\varepsilon, \eta,k}_t-\mathcal{L}_t \varphi ^{\varepsilon, \eta, k}>0 ,\ \ \ \textrm{on~} \overline{B(T,x_0,\varepsilon)}.$$
We would like to emphasize that, for convenience, we work here with the norm
$$\|(t,x)\|=\max \{|t|, \|x\|\}, (t,x)\in \mathbb{R}\times \mathbb{R}^d,$$
where $\|x\|$ is the Euclidean norm on $\mathbb{R}^d$. With this particular norm, we can 
use \eqref{bound} to obtain
$$\varphi^{\varepsilon, \eta,k}\geq \varepsilon +v \textrm{~on~} (\overline{B(T,x_0,\varepsilon)}-B(T,x_0,\varepsilon /2))\cap( [0,T]\times \mathbb{R}^d).$$
Also, $$\varphi^{\varepsilon, \eta,k}(T,x)\geq v^+(T,x_0)\geq g(x)+\varepsilon \textrm{~for~} |x-x_0|\leq \varepsilon$$ and
 $$\varphi^{\varepsilon, \eta,k}(t,x)\geq v^+(T,x_0)\geq l(t,x)+\varepsilon\; \textrm{~for~} \max\{|x-x_0|, T-t\}\leq \varepsilon. $$
 Now, we can choose $\delta <\varepsilon$ and define 
$$v^{\varepsilon, \eta,k, \delta }=
\left \{
\begin{array}{l}
 v\wedge \Big ( \varphi ^{\varepsilon, \eta,k}-\delta  \Big)\textrm{~on~} \overline{ B(T, x_0, \varepsilon)},\\
v \textrm{~outside~}\overline{ B(T, x_0, \varepsilon)}.
\end{array}
\right.
$$
Using again the ideas in 
 items 1-3 in Step 1 of the proof, we can show that $v^{\varepsilon, \eta, k, \delta}\in \mathcal{V}^+$ but $v^{\varepsilon, \eta, k, \delta}(T,x_0)=v^+(T,x_0)-\delta <v^+(T,x_0)$, leading to a contradiction.

\hfill $\square$. 

\subsection{Verification by comparison}
\begin{Definition}\label{comparison} We say that the viscosity comparison principle holds for the equation \eqref{pde} if,
whenever we have  a bounded, upper semi-continuous (USC) sub-solution $v$ of  \eqref{pde1}
and a bounded lower semi-continuous (LSC) super-solution $w$ of \eqref{pde2}
 then 
$v\leq w.$
\end{Definition}
\begin{Theorem}\label{theorem}
Let $l,u,g$ be bounded such that $l$ is USC, $u$ is LSC and $g$ is continuous.
Assume that conditions \eqref{u} and \eqref{l} hold. Assume also that the comparison principle  is satisfied. Then  there exists a unique bounded and continuous viscosity solution $v$ to \eqref{pde} which equals both the lower and the upper  values of the game which means
$$v_*=v=v^*.$$
In addition,
 for each $(s,x)\in[0,T]\times \mathbb{R}^d$, 
 the stopping times
 $$ \rho ^*(s,x)=\rho ^+(v,s,x,s) \textrm{~and~}
\tau ^*(s,x) =\tau ^-(v,s,x,s)
$$
 are optimal for the two players.
 
 \end{Theorem}
 \proof{It is clear that the unique and continuous viscosity solution of \eqref{pde} is 
 $$v^-=v_*=v^*=v^+.$$ 
 The only thing to be proved is that 
$$\rho ^*(s,x)=\rho ^+(v,s,x,s),\ \ \ \tau ^*(s,x) =\tau ^-(v,s,x,s)$$
are  optimal strategies.
Let $v_n$ an increasing sequence such that $v_n\in \mathcal{V}^-$ and
$v^-=\sup _n v_n$. According to the proof of Lemma \ref{lem:verification}, if we define
$$\tau ^-_n:=\tau ^-(v_n, s,x,s),$$
then, for each $\rho \in \mathcal{T}^{s,x}$ we have
$$v_n(s,x)\leq \mathbb{E}\left [v_n(\tau ^-_n\wedge \rho, X^{s,x}_{\tau ^-_n\wedge \rho})\right]\leq J(s,x, \tau_n ^-,\rho).$$
Using the definition of $J$, together with the fact that the lower obstacle $l$ is USC, $l\leq u$ and $l(T,\cdot)\leq g\leq u(T,\cdot)$, we can pass to the limit to conclude
$$v^- (s,x)\leq  J(s,x, \tau ^-_{\infty},\rho), \ \textrm{~for~all~}\rho  \in \mathcal{T}^{s,x},$$
where $$\tau ^-_{\infty}:=\lim _n \tau ^- _n \leq \tau ^*.$$
The inequality $\tau ^-_n \leq \tau ^*$ is obvious taking into account the definition of the two stopping times and the fact that $v_n\leq v$. Now, since $v_n$ are LSC (actually continuous) and converge increasingly to the \emph{continuous} function $v=v^-$, then we can use  Dini's Theorem to conclude that the convergence is uniform on compacts $|x|\leq N$.
Now, we have 
$$v_n(\tau ^-_n, X^{s,x}_{\tau^- _n})\leq l(\tau ^-_n, X^{s,x}_{\tau^- _n}) \textrm{~on~} \{\tau ^-_n<T\} \supset   \{\tau^- _{\infty}<T\} $$
Since $v_n\nearrow v$ uniformly on compacts, and $l$ is USC we can pass to the limit, first in $n$ to conclude that
$$v(\tau ^-_{\infty}, X^{s,x}_{\tau^- _{\infty}})\leq l(\tau ^-_{\infty}, X^{s,x}_{\tau^- _{\infty}}) \textrm{~on~}   \{\tau^- _{\infty}<T\} \cap\{|X^{s,x}_t|\leq N\ (\forall)\ t\},$$
and then let $N\rightarrow \infty$ to obtain
$$v(\tau ^-_{\infty}, X^{s,x}_{\tau^- _{\infty}})\leq l(\tau ^-_{\infty}, X^{s,x}_{\tau^- _{\infty}}) \textrm{~on~}   \{\tau^- _{\infty}<T\}.$$
This shows that, $\tau ^*\leq \tau ^-_{\infty}$, so $\tau ^*= \tau ^-_{\infty}$. 
To sum up, we have that
$$v(s,x)\leq  J(s,x, \tau ^*,\rho), \ \textrm{~for~all~}\rho  \in \mathcal{T}^{s,x}.$$
Similarly, we can prove
$$v(s,x)\geq  J(s,x, \tau ,\rho^*), \ \textrm{~for~all~}\tau  \in \mathcal{T}^{s,x},$$
and the proof is complete, showing that $\rho^*$ and $\tau ^*$ are optimal strategies for the two players.} 
\endproof
\begin{Remark} {\rm The regularity needed in our main results, Theorems \ref{Perron} and \ref{theorem}, is the minimal regularity needed to ensure that the stopping regions are closed and the continuation regions are open.
However, unless one has a direct way to check \eqref{u} and \eqref{l}, the opposite semi-continuity of the obstacles has to be assumed, according to Remark \ref{rem:nonempty} . Theorem \ref{theorem} can be therefore applied in general when the obstacles are continuous.}
\end{Remark}
 \section{Optimal stopping and (single) obstacle problems }
 The case of optimal stopping can be easily treated as above, assuming, \emph{formally} that the upper obstacle is infinite
 $u=+\infty.$
 However, in order to have meaningful results, Definitions \ref{def:supersolutions} and \ref{def:subsolutions} have to be modified accordingly, together with the main Theorems \ref{Perron} and \ref{theorem}, noting that
$\rho^-=T,$
i.e. the minimizer player never stops the game before maturity.

Having this in mind, a stochastic super-solution, is, therefore a \emph{continuous excessive} function in the lingo of optimal stopping for Markov processes. We remind the reader that we did \emph{not} assume the selection $X^{s,x}$ to be Markov, though.

One could show easily that $v^+\in \mathcal{V}^+$ if the continuity in the definition of $\mathcal{V}^+$ were relaxed to USC. Under this USC (only) framework, 
the corresponding  Theorem \ref{Perron}, part (i) can be translated as: \emph{the least  excessive function is a viscosity sub-solution of the obstacle problem}. 
While it is known, in the Markov  case that the  value function is the least excessive function (see  \cite{shiryaev}), we \emph{never} analyze the value function \emph{directly} in our approach.
 We can not show directly that $v^-\in \mathcal{V}^-$. Viscosity comparison is necessary to close the argument.

 All the proofs of the modified Theorems \ref{Perron} and \ref{theorem} work identically, up to obvious modifications.
\bibliographystyle{amsplain}

\bibliography{references} 

\end{document}